\documentclass[11pt]{amsart}

\usepackage[english]{babel}
\usepackage[utf8]{inputenc}
\usepackage[T1]{fontenc}
\usepackage{lmodern}
\usepackage{microtype}
\usepackage{indentfirst}
\usepackage{geometry}
\usepackage{enumitem}
\usepackage{amsmath, amssymb, amsthm}
\usepackage{eucal, mathrsfs, bm, bbm}
\usepackage{stmaryrd,mathtools,braket}
\usepackage[all]{xy}
\usepackage{xspace}
\usepackage[foot]{amsaddr}
\usepackage{hyperref}
\usepackage[capitalise,noabbrev]{cleveref}

\hypersetup{
urlbordercolor=1 1 1,
linkbordercolor=0 0 1
}

\theoremstyle{plain}
\newtheorem{theor}{Theorem}[section]
\newtheorem{prop}[theor]{Proposition}
\newtheorem{lem}[theor]{Lemma}
\newtheorem{corol}[theor]{Corollary}
\newtheorem{defin}[theor]{Definition}
\theoremstyle{definition}
\newtheorem{rem}[theor]{Remark}

\crefname{defin}{Definition}{Definitions}
\crefname{rem}{Remark}{Remarks}
\crefname{theor}{Theorem}{Theorems}
\crefname{prop}{Proposition}{Propositions}
\crefname{lem}{Lemma}{Lemmas}
\crefname{corol}{Corollary}{Corollaries}
\crefname{claim}{Claim}{Claims}
\crefname{exmp}{Example}{Examples}
\newcommand{\abscatm}[1]{\mathbb{#1}}		
\newcommand{\catcm}{\abscatm{C}}		\newcommand{\catct}{$\catcm$\xspace}
\newcommand{\catpm}{\abscatm{P}}		\newcommand{\catpt}{$\catpm$\xspace}
\newcommand{\catem}{\abscatm{E}}		\newcommand{\catet}{$\catem$\xspace}
\newcommand{\catmm}{\abscatm{M}}		\newcommand{\catmt}{$\catmm$\xspace}
\newcommand{\conccatm}[1]{\mathsf{#1}}		
\newcommand{\cattopm}{\conccatm{Top}}		\newcommand{\cattopt}{$\cattopm$\xspace}
		
\newcommand{\subm}{\textup{Sub}}	
\newcommand{\indxup}[2]{#1_{\textup{#2}}}
\newcommand{\fnctidxup}[1]{\Gamma_{\!\textup{#1}}}
\newcommand{\excomm}[1]{\indxup{#1}{ex}}	
\newcommand{\excomcm}{\excomm{\catcm}}		\newcommand{\excomct}{$\excomcm$\xspace}
\newcommand{\exfnctm}{\fnctidxup{ex}}

\newcommand{\posrefm}[1]{\indxup{#1}{po}}

\newcommand{\posrefpm}[1]{\posrefm{(\catpm/#1)}}
\DeclareMathOperator{\hocatm}{Ho}
\newcommand{\hocatcm}{\hocatm \catcm}	\newcommand{\hocatct}{$\hocatcm$\xspace}
\newcommand{\hocatmm}{\hocatm \catmm}	\newcommand{\hocatmt}{$\hocatmm$\xspace}
\newcommand{\fibcatm}[1]{{#1}_{\textup{f}}}
\newcommand{\fibcatmm}{\fibcatm{\catmm}}	\newcommand{\fibcatmt}{$\fibcatmm$\xspace}
\newcommand{\toop}{\leftarrow}
\newdir{|>}{!/4.5pt/@{|}*:(1,-.2)@^{>}*:(1,+.2)@_{>}*!/-3.5pt/@{ }}
\newdir{ (}{!/ 2.5pt/@{ }*!/:(-.1,-1) 5.7pt/@{(}}
\newdir{ >}{!/2pt/@{ }*@{>}}
\newcommand{\xyar}[5]{\ar@{#1}|<{#2}#3|>{\SelectTips{#4}{}\object@{#5}}}
\newcommand{\xys}[2][]{\xyar{-}{#1}{#2}{cm}{>}}
\newcommand{\xydot}[2][]{\xyar{.}{#1}{#2}{cm}{>}}

\newcommand{\mono}{\hookrightarrow}

\newcommand{\xymonor}[1]{\xyar{^{(}-}{}{#1}{cm}{>}@<-1pt>}

\newcommand{\repi}{\, \mbox{\large$\rightarrowtriangle$}\, }

\newcommand{\xyrepi}[2][]{\xyar{-}{#1}{#2}{xy}{|>}}
\newcommand{\psrel}{\rightrightarrows}

\newcommand{\xymapsto}[1]{\xyar{-}{\object@{|}}{#1}{cm}{>}}
\newcommand{\fibr}{\twoheadrightarrow}

\newcommand{\xyfibr}[2][]{\xyar{-}{#1}{#2}{cm}{>>}}
\newcommand{\cofibr}{\rightarrowtail}

\newcommand{\xycofibr}[1]{\xyar{-}{\SelectTips{cm}{}\object@{ >}}{#1}{cm}{>}}
\newcommand{\wequivand}[1]{\overset{\text{\tiny$ \sim $}}{#1}}
\newcommand{\wequiv}{\wequivand{\rightarrow}}

\newcommand{\xywequiv}{\object@{~}}
\newcommand{\xycenterm}[2][=2em]{\vcenter{\hbox{\xymatrix@#1{#2}}}}
\newcommand{\xycentermqed}[2][=2em]{\begin{gathered}[b]\xycenterm[#1]{#2}\\[-\dp\strutbox]\end{gathered}\qedhere}
\newcommand{\id}{\textup{id}}
\newcommand{\pbkun}[1]{\langle #1 \rangle}
\newcommand{\htpc}{\simeq}
\newcommand{\limpd}{\Rightarrow}
\newcommand{\liff}{\Leftrightarrow}
\newcommand{\allw}{\forall^{\textup{w}}}
\newcommand{\typing}{:}
\newcommand{\ie}{\textit{i.e.}\xspace}

\renewcommand{\iff}{if and only if\xspace}
\newcommand{\twothree}{2-out-of-3\xspace}

\newcommand{\wrt}{with respect to\xspace}
\newcommand{\mltt}{Martin-L\"of type theory\xspace}

\makeatletter
\@namedef{subjclassname@2020}{\textup{2020} Mathematics Subject Classification}
\makeatother

\title{The Fullness Axiom and exact completion of homotopy categories}
\author{Jacopo Emmenegger}
\email{\href{mailto:op.emmen@gmail.com}{op.emmen@gmail.com}}
\address{School of Computer Science, University of Birmingham, Birmingham B15 2TT, UK.}

\begin{document}

\begin{abstract}
We use a category-theoretic formulation of Aczel's Fullness Axiom from Constructive Set Theory
to derive the local cartesian closure of an exact completion.
As an application, we prove that
such a formulation is valid in the homotopy category of any model category satisfying mild requirements,
thus obtaining in particular the local cartesian closure of the exact completion
of topological spaces and homotopy classes of maps.
Under a type-theoretic reading, these results provide a general motivation
for the local cartesian closure of the category of setoids.
However, results and proofs are formulated solely in the language of categories,
and no knowledge of type theory or constructive set theory is required on the reader's part.
\end{abstract}

\maketitle

\section{Introduction}

In the paper that generalises the exact completion construction to an arbitrary category with
\emph{weak} finite limits, where a universal arrow is not required to be unique,
Carboni and Vitale advocated a deeper study of that construction
applied to homotopy categories~\cite{CarboniVitale1998}.
These categories, indeed, form a large class of natural examples
of categories with weak finite limits,
in the sense that they do not arise as projective covers of finitely complete categories.
A first step in this direction was made by Gran and Vitale in~\cite{GranVitale1998},
where they provide a complete characterisation of those 
exact completions of categories with weak finite limits
(henceforth ex/wlex completions) that produce a pretopos,
and apply this result to show that the exact completion
of the category of topological spaces and homotopy classes of maps
is indeed a pretopos.
However, the problem of determining whether
it is also locally cartesian closed is explicitly left open.

The author has given a complete characterisation
of locally cartesian closed ex/wlex completions in~\cite{Emmenegger2018}.
That characterisation is however not much suited to
the study of the ex/wlex completion of a homotopy category \hocatmt,
when instead a formulation in terms of the original
Quillen model category \catmt would be preferable.
The present paper provides a condition
ensuring the local cartesian closure of
the ex/wlex completion $ \excomm{(\hocatmm)} $
for a large class of model categories.
Somewhat surprisingly, this condition turns out to be
what Carboni and Rosolini named weak local cartesian closure in~\cite{CarboniRosolini2000},
that is, simply existence of weak dependent products.

As we shall prove in the last section,
the homotopy quotient of a weak dependent product in \catmt
is a \emph{dependent full diagram} in $ \hocatmm $.
The latter is a generalisation to arbitrary categories with weak finite limits
of a concept introduced in~\cite{EmmeneggerPalmgren2017}
to prove the local cartesian closure of the exact completion
of a well-pointed category with finite products and weak equalisers.
Anther precursor is the axiom F for a class of small maps
in~\cite{vdBergMoerdijk2008}, which is proved to be stable under ex/reg completions.
We shall comment on the (tight) relation between this axiom
and dependent full diagrams in \cref{rem:axF}.
Indeed, both the universal property of dependent full diagrams
and axiom F are inspired by Aczel's Fullness Axiom from
the constructive set theory CZF~\cite{Aczel1978,AczelRathjen2001}.
The Fullness Axiom is a collection principle
asserting the existence of what Aczel calls \emph{full sets},
that is, sets containing enough total relations (a.k.a.\ multi-valued functions).
This axiom is strictly weaker than the Power Set Axiom
(and it is regarded as a predicative principle)
but strong enough, in particular, to entail the Exponentiation Axiom,
which asserts that functions between two sets form a set.

We prove that also in the general case of an ex/wlex completion,
existence of dependent full diagrams in \catct is enough
to derive the local cartesian closure of \excomct.
That this should be possible follows from the observation, due to Erik Palmgren,
that arrows $ V \to Y $ out of a weak product $ Z \toop V \to X $ may be understood
as families indexed on $ Z $ of multi-valued functions from $ X $ to $ Y $.
A more robust formulation of this observation is in \cref{rem:totrel}.
In addition, dependent full diagrams endow the internal logic with universal quantification and implication
which, in turn, can be used to extract from multi-valued functions only the functional ones (cf.~\cref{lem:wcc}).
At this point it is enough to construct a suitable equivalence relation to obtain an exponential in \excomct.

In order to prove that dependent full diagrams are homotopy quotients of weak dependent products,
we exploit the concepts of \emph{path category} and \emph{weak homotopy $\Pi$-type},
recently introduced by van den Berg and Moerdijk in~\cite{vdBergMoerdijk2018}.
A path category is a slight strengthening of Brown's fibration category.
In particular, the subcategory $ \fibcatmm $ on the fibrant objects
of a model category \catmt is a path category
as soon as all the objects of \catmt are cofibrant.
Weak homotopy $\Pi$-types in a path category \catct are what van den Berg and Moerdijk use
to derive the local cartesian closure of $ \excomm{(\hocatcm)} $,
the homotopy exact completion of \catct.
We show that if \catmt is right proper,
then weak homotopy $\Pi$-types arise as fibrant replacements of weak dependent products,
so that a weak dependent product in \catmt gives rise to a weak homotopy $\Pi$-type in $ \fibcatmm $.
Furthermore, in the same way as pullbacks along fibrations
enjoy the additional universal property of homotopy pullbacks,
also weak homotopy $\Pi$-types enjoy an additional universal property up to homotopy,
which shows that the homotopy quotient
maps weak homotopy $\Pi$-types to dependent full diagrams.

Under a type-theoretic reading, the results in the present paper provide
a general motivation for the local cartesian closure of the category of setoids in \mltt.
Indeed, the category of contexts of \mltt
is a path category~\cite{vdBerg2018a}, see also~\cite{GambinoGarner2008},
and $ \prod $-types endow it with weak homotopy $\Pi$-types.
More generally, we obtain a more elementary proof of the local cartesian closure
of a homotopy exact completion.
It should be noted that,
under the reading of arrows out of a weak product as multi-valued functions,
single-valued functions in a homotopy category $ \hocatcm $
appear as ``homotopy-irrelevant'' arrows.
Indeed, an arrow $ k $ out of a homotopy limit,
say a homotopy pullback of $ f \colon X \to Y $ and $ g \colon Z \to Y $,
induce an arrow in $ \excomm{(\hocatm \catcm)} $
out of the actual pullback of $ f $ and $ g $ in
\iff values of $ k $ only depend on pairs $ (x,z) $
and not on the homotopy witnessing $ f(x) \htpc g(z) $.
The analogy with the role of homotopy-irrelevant fibrations in
the argument for the local cartesian closure of $ \excomm{(\hocatm \catcm)} $
from van den Berg and Moerdijk~\cite{vdBergMoerdijk2018}
may be worth further investigation.
Indeed, in type-theoretic terminology,
these are the proof-irrelevant setoid families
whose importance has been stressed by Palmgren~\cite{Palmgren2012a}.

Furthermore, the observation that dependent full diagrams
naturally arise as homotopy quotients of weak dependent products shows that
existence of the former is not just a particular feature of the category of types in \mltt,
the only example in~\cite{EmmeneggerPalmgren2017}.
On the contrary,
it provides a large class of examples of categories with weak finite limits and dependent full diagrams.
In particular, we obtain the local cartesian closure of the exact completion
of the category of spaces and homotopy classes of maps,
thus answering a question left open in~\cite{GranVitale1998}.

The first half of the paper is devoted to the proof that existence of dependent full diagrams
imply the local cartesian closure of the exact completion.
In order to simplify the presentation,
we split the argument in two steps.
In \cref{sec:fulldiag}, after a brief recap on ex/wlex completions,
we define a non-indexed version of a full diagram in \catct and,
assuming that \excomct (or, equivalently, \catct)
has the needed structure for implication and universal quantification,
we construct from it an exponential in \excomct.
\Cref{sec:dfulldiag} contains the definition of the more general dependent full diagrams
and the proof that their existence gives rise
both to right adjoints to inverse images and to non-indexed full diagrams.
Finally, \cref{sec:hocat} covers the case of homotopy categories.

\section{Full diagrams}
\label{sec:fulldiag}

We briefly recall some basic facts about weak limits and the exact completion
which are essential to our treatment.
For additional background notions and notations we refer to~\cite{CarboniVitale1998,CarboniRosolini2000}
and to section 1 of~\cite{Emmenegger2018}.
In this section and the next one
\catct denotes a category with weak finite limits
and \excomct denotes its ex/wlex completion,
that we shall refer to as the exact completion of \catct.
Regular epis are denoted with a triangle head, like in $ A \repi B $,
while hook arrows $ A \mono B $ denote monos.

Recall that weak limits are defined as usual limits but without requiring
uniqueness of a universal arrow.
An arrow $ f \colon V \to Y $ in \catct
from a weak product $ Z \toop V \to X $ of $ Z $ and $ X $
is \emph{determined by projections}~\cite{Emmenegger2018}
if it coequalises every pair of arrows jointly coequalised
by the two weak product projections.
A \emph{weak exponential}~\cite[Definition 2.1]{CarboniRosolini2000}
in \catct of two objects $ Y $ and $ X $ consists of an object $ W $,
a weak product $ W \toop V \to X $ and an arrow $ V \to Y $
which is determined by projections,
such that for any object $ W' $,
weak product $ W' \toop V' \to X $
and arrow $ V' \to Y $ determined by projections,
there are (not necessarily unique) arrows
$ h \colon W' \to W $ and $ k \colon V' \to V $
making the obvious diagram commute.

An object $ X $ in a category \catet is called \emph{(regular) projective}
if, for every regular epi $ g \colon A \repi B $ and arrow $ f \colon X \to B $,
there is a \emph{lift} of $ f $ against $ g $,
\ie an arrow $ f' \colon X \to A $ such that $ g f' = f $.
A \emph{projective cover} of an exact category \catet is a full subcategory
\catpt such that
\textit{(i)} every object in \catpt is projective in \catet, and
\textit{(ii)} every object in \catet is covered by an object in \catpt,
\ie for every $ A $ in \catet there are $ X $ in \catpt
and a regular epi $ X \repi A $.
\catet has \emph{enough projectives} if it has a projective cover.
Whenever we are given a projective cover \catpt of an exact category \catet,
we adopt the convention of using letters from $ P $ to $ Z $
for objects in \catpt.
The reader should however keep in mind that
\catpt is not in general closed under limits that exist in \catet.

Recall from~\cite{CarboniVitale1998} that,
for a category \catct with weak finite limits,
the exact completion \excomct can be described as the category whose objects
are pseudo equivalence relations $R \psrel X$ in \catct
and whose arrows from $R \psrel X$ to $S \psrel Y$  are 
equivalence classes of those arrows $X \to Y$ of \catct
that map related elements to related elements,
where $f,f' \colon X \to Y$ are equivalent if they are related in $S \psrel Y$.
The full subcategory of \excomct on the free pseudo equivalence relations,
\ie those with equal legs,
is a projective cover of \excomct.
It is equivalent to \catct via
the embedding of \catct into \excomct,
that maps an object $X$ to the pair of identities $\id_X, \id_X \colon X \psrel X$.
Conversely, a projective cover \catpt of an exact category \catet
has weak finite limits and
its exact competion $ \excomm{\catpm} $
is equivalent to \catet~\cite[Thm.\ 16]{CarboniVitale1998}.
A weak limit in \catpt is obtained covering with a projective
the corresponding limit in \catet.
More generally,
a cone in \catpt over a diagram $\mathcal{D}$ in \catpt is a weak limit
\iff the unique arrow from the cone
into the limit in \catet of $\mathcal{D}$ is a regular epi.
see~\cite[Lemma 1.7]{Emmenegger2018}.

Given a projective cover \catpt of \catet exact and a weak product
$Z \overset{p_1}{\longleftarrow} V \overset{p_2}{\longrightarrow} X$ in \catpt,
an arrow $ f \colon V \to Y $ is determined by projections in \catpt \iff
it factors in \catet through the regular epi
$ \pbkun{p_1,p_2} \colon V \repi Z \times X $.
It follows that \catpt has weak exponentials if \catet is cartesian closed:
given $ Y,X \in \catpm $, one just need cover $ Y^X $ with a projective $ W $
and to do the same with $ W \times X $.
However, as argued in the last section of~\cite{Emmenegger2018},
the universal property of weak exponentials does not seem to be suited
to prove the cartesian closure of \excomct
when \catct only has weak finite limits.
A complete characterisation in terms of what we called
\emph{extensional simple products} is presented in~\cite{Emmenegger2018},
but here we look at yet another (weakly) universal property.

\begin{rem} \label{rem:totrel}
Let \catet be exact with a projective cover \catpt
and let $Z,X,Y$ be three objects in \catpt.
There is an isomorphism between the poset of subobjects
$ \subm_{\catem}(Z \times X \times Y) $
and the poset reflection $ \posrefpm{(Z,X,Y)} $
of the category of spans
\[
\xycenterm[=1em]{
&	V	\xys{}[dl] \xys{}[d] \xys{}[dr]
&\\
Z	&	X	&	Y	}
\]
in \catpt
over $Z,X$ and $Y$~\cite[Lemma 35]{CarboniVitale1998}.
This isomorphism restricts between those subobjects
$ R \mono Z \times X \times Y $ such that $ R \to Z \times X $ is regular epic,
and those (equivalence classes of) spans
such that $ Z \toop V \to X $ is a weak product.
This restricts further between those subobjects
such that $ R \to Z \times X $ is iso,
\ie essentially graphs of arrows $Z \times X \to Y$,
and those spans
such that $ Z \toop V \to X $ is a weak product and
$ V \to Y $ is determined by projections.
\end{rem}

The previous remark allows us to
understand arrows $ f \colon V \to Y $
in a category \catct with weak finite limits,
where $ Z \toop V \to X $ is a weak product, as families,
indexed by $ Z $, of total relations
(\ie multi-valued functions) from $ X $ to $ Y $.
Such a total relation is functional (\ie a single-valued function)
precisely when $ f $ is determined by projections.
This reading suggests that,
in order to have a suitable universal property with respect to arbitrary arrows
$ V \to Y $ out of a weak product,
we should look for some property of closure
with respect to (families of) total relations from $ X $ to $ Y $.
A promising notion,
that indeed proves to be useful, is that of a \emph{full set}
from the constructive set theory CZF.

A set $ f $ is full for two sets $ a $ and $ b $
if it consists of multi-valued functions from $ a $ to $ b $ and,
for every multi-valued function $ r $ from $ a $ to $ b $,
there is $ s \in f $ such that $ s \subseteq r $.
The \emph{Fullness Axiom} states that,
for any two sets $ a $ and $ b $, there is a set which is full for $ a $ and $ b $.
This axiom was introduced in the context of
Constructive Zermelo-Fraenkel set theory (CZF) by Peter Aczel in~\cite{Aczel1978}
in order to provide a simpler formulation of the axiom schema of Subset Collection.
This axiom implies in particular the so-called Exponentiation Axiom,
that the class of functions between two sets is a set.
The next definition is inspired by the notion of full set.

\begin{defin} \label{def:fullobj}
Let $ X $ and $ Y $ be two objects in a category \catct with weak finite limits.
A \emph{full diagram from $ X $ to $ Y $} consists of a weak product
$ U \overset{p_1}{\longleftarrow} V \overset{p_2}{\longrightarrow} X $
and an arrow $ f \colon V \to Y $
such that,
for every object $ U' $, weak product $ U' \toop V' \to X $
and arrow $ f \colon V' \to Y $,
there are an arrow $ h \colon U' \to U $,
a weak pullback $ U' \toop P \to V $ of $ U' \to U \toop V $ and
an arrow $ k \colon P \to V' $ such that the diagram below commutes.

\begin{equation} \label{fullobj:univ}
\xycenterm{
&	U'	\xydot{}[dl]_-h
&&	P	\xydot{}[ll] \xydot{}[dl] \xydot{}[r]_-k
&	V'	\xys{}@/_1.5em/[lll]
		\xys{|!{[dl];[dr]}\hole}@/^/[ddll]^(.7){f'}
		\xys{}@/^/[dr]
&\\
U	&&	V	\xys{}[ll]^-{p_1} \xys{}[d]_-f \xys{}[rrr]_(.7){p_2}
&&&	X
\\
&&	Y	&&&}
\end{equation}
\end{defin}

\begin{rem} \label{rem:fullobj}\hfill
\begin{enumerate}
\item
The notion of full diagram is independent of the specific weak product in the following sense.
If the pair $ U \toop V \to X $, $ V \to Y $ is a full diagram from $ X $ to $ Y $,
then any other weak product $ U \toop W \to X $ together with the composite $ W \to V \to Y $
is a full diagram from $ X $ to $ Y $.
\item
In the case of a projective cover \catpt of an exact category \catet,
diagram~\eqref{fullobj:univ} in \cref{def:fullobj} can be written in \catet as

\[
\xycenterm[R=2.5em@C=3.5em]{
U' \times X	\xys{}[d]_-{h \times X}
&	P	\xyrepi{}[l] \xys{}[d] \xys{}[r]_-k
&	V'	\xyrepi{}@/_1.5em/[ll] \xys{}[d]^-{f'}
\\
U \times X	&	V	\xyrepi{}[l]^-{\pbkun{p_1,p_2}} \xys{}[r]_-f
&	Y	}
\]
where the left-hand square is covering,
\ie the induced arrow from $ P $ to the pullback
of $\pbkun{p_1,p_2}$ and $h \times X$ is a regular epi.
\end{enumerate}
\end{rem}

\begin{lem} \label{lem:wexpfull}
Suppose that \catct has binary products.
Then weak exponentials are full diagrams
and any full diagram from $X$ to $Y$
gives rise to a weak exponential of $Y$ and $X$.
\end{lem}

\begin{proof}
Since \catct has binary products, every weak product retracts onto the product of the same objects.
Also, an arrow from a weak product is determined by projections \iff
it factors through the retraction onto the product.
It follows that we may assume that the domain of a weak evaluation is a product, rather than just a weak product.

Let then $ W $ be a weak exponential of $ Y $ and $ X $
with weak evaluation $ e \colon W \times X \to Y $.
Given $ f \colon V \to Y $ from a weak product $ U' \toop V \to X $
we can take $ U \times X $ as $ P $ in \eqref{fullobj:univ}:
the arrow $ k \colon U \times X \to V $
is a section of the retraction $ V \to U \times X $
and the arrow $ h \colon U \to W $
is obtained by the universal property of the weak exponential
applied to the composite
$ f k \colon U \times X \to Y $.

For the converse,
let $ U \toop V \to X $, $ f \colon V \to Y $ form a full diagram
and let $ s \colon U \times X \mono V$
be a section of the retraction $ V \to U \times X $.
We shall prove that $ f s \colon U \times X \to Y $ exhibits $ U $
as a weak exponential of $ Y $ and $ X $.
Given $ f' \colon U' \times X \to Y $,
the universal property of full diagrams yields an arrow $ h \colon U' \to U $
and a commutative diagram
\[
\xycenterm[C=3em]{
U' \times X	\xys{}[d]_{h \times X}
&	P	\xys{}[l] \xys{}[d] \xys{}[r]
&	U' \times X	\xys{}[d]^{f'}
\\
U \times X	&	V	\xys{}[l] \xys{}[r]^-f
&	Y	}
\]
where the left-hand square is a weak pullback.
It follows that $ P \to U' \times X $ has a section $ s' $ such that the diagram
\[
\xycenterm{
U' \times X\,	\xys{}[d]_{h \times X} \xymonor{}[r]^-{s'}
&	P	\xys{}[d]
\\
U \times X\,	\xymonor{}[r]^-s	&	V	}
\]
commutes.
The equation $ f s (h \times X) = f' $ then follows immediately.
\end{proof}

In particular, a category with finite limits has full diagrams \iff it has weak exponentials.
\Cref{lem:wcc} shows that,
whenever the internal logic of \excomct (equivalently, of \catct)
supports implication and universal quantification,
the left-to-right implication also holds when \catct only has weak finite limits.

First, recall that descent in exact categories allows us to prove the following,
where $\posrefm{\abscatm{X}}$ denotes the poset reflection of the category $\abscatm{X}$.
See also \cite[Remark 1.9]{Emmenegger2018}.

\begin{lem} \label{lem:deslog}
\excomct has right adjoints to inverse images \iff
\catct has right adjoints to weak pullback functors,
\ie the functors
$ \posrefm{(\catcm/X)} \to \posrefm{(\catcm/Y)} $
induced by weak pullback along arrows $ f \colon Y \to X $.
\end{lem}

\begin{proof}
One direction follows by the natural isomorphisms
$ \subm_{\excomcm}(\exfnctm X) \cong \posrefm{(\catcm/X)} $.
For the other direction apply Theorem 2 in Section 3.7 of~\cite{BarrWells1985}.
\end{proof}

We now need a lemma which, for convenience,
we formulate using an exact category \catet with a fixed projective cover \catpt.

\begin{lem} \label{lem:wcc}
Let \catet be an exact category with a projective cover \catpt
and suppose that \catet has right adjoints to inverse images along any arrow.
If \catpt has full diagrams, then for every $ X $ in \catpt and $ B $ in \catet
there are an object $ W $ in \catpt and an arrow $ W \times X \to B $ in \catet
which are weakly terminal \wrt objects $ Z $ in \catpt and arrows $ Z \times X \to B $ in \catet.
\end{lem}

\begin{proof}
Let $ b \colon Y \repi B $ be a cover of $ B $ with $ Y $ in \catpt,
and take a full diagram
$ U \overset{p_1}{\longleftarrow} V \overset{p_2}{\longrightarrow} X $, $ f \colon V \to Y $.
The idea is to extract from $ U $ (codes of) functional relations.
Let $ \gamma = \pbkun{\gamma_1,\gamma_2,\gamma_3} \colon I \mono U \times X \times B $
be the image factorisation of
$ \pbkun{p_1,p_2,bf} \colon V \to U \times X \times B $
and denote with $ \phi \colon F \mono U $
the subobject defined by the formula in context
\[
u \typing U \ |\ (\forall x \typing X)(\forall y,y' \typing B)\, \gamma(u,x,y) \land \gamma(u,x,y') \limpd y = y'.
\]
In other words,
given an arrow $ a \colon A \to U $ in \catet, consider the diagram
\[
\xycenterm[=1.5em]{
H	\xyrepi{}[dd]  \xyrepi{}[dr] \xys{}[rr]
&&	K	\xyrepi{|!{[dl];[dr]}{\hole}}[dd] \xyrepi{}[dr]
&\\
&	A \times_U I	\xyrepi{|(.65){A \times \gamma_2}}[dd] 
			\xys{}[rr]_(.7){\pi_2}
&&	I	\xyrepi{}[dd]^-{\pbkun{\gamma_1,\gamma_2}}
\\
A \times_U I	\xyrepi{}[dr]_-{A \times \gamma_2}
		\xys{|!{[ur];[dr]}{\hole}}[rr]^-(.75){\pi_2}
&&	I	\xyrepi{|(.4){\pbkun{\gamma_1,\gamma_2}}}[dr]
&\\	
&	A \times X	\xys{}[rr]_{a \times X}
&&	U \times X	}
\]
where all the squares are pullback.
Then
\begin{equation} \label{wcc:wedef}
a \leq \phi \ \liff \
A \times_U I \overset{\pi_2}{\longrightarrow} I \overset{\gamma_3}{\longrightarrow} B
\text{ coequalises } H \psrel A \times_U I.
\end{equation}

The existence of an arrow $ e \colon F \times X \to B $
follows from~\eqref{wcc:wedef} taking $ a = \phi $.
We shall show that $ e $ satisfies the required universal property.
It then follows easily that $ W \times X \to B
$ satisfies it as well for any cover $ W \repi F $.
In particular, \catpt will have weak exponentials.

Given $ g \colon Z \times X \to B $ with $ Z \in \catpm $,
let $ g' \colon V' \to Y $ be a cover of $ g $,
\ie be such that the right-hand square in diagram~\eqref{wcc:rcuniv} is covering.
By the universal property of a full diagram,
we have, in particular, an arrow $ h \colon Z \to U $ and a commutative diagram
\begin{equation} \label{wcc:rcuniv}
\xycenterm[R=2.5em@C=4em]{
Z \times X	\xys{}[d]_{h \times X}
&	P	\xyrepi{}[l] \xys{}[d] \xyrepi{}[r]
&	Z \times X	\xys{}[d]^g
\\
U \times X
&	V	\xyrepi{}[l]_-{\pbkun{p_1,p_2}} \xys{}[r]^-{bf}
&	B	}
\end{equation}
where the left-hand square is covering.
It follows that
the induced arrow $q \colon P \to Z \times_U I$ is a regular epi.
Consider now the solid arrows in diagram~\eqref{wcc:weuniv} below.
The two right-hand squares are pullback
and the front left-hand square commutes by definition of $e$.
\begin{equation} \label{wcc:weuniv}
\xycenterm[R=1.5em@C=2.5em]{
&&&	Z \times_U I	\xyrepi{|(.35){Z \times \gamma_2}|!{[dl];[dr]}{\hole}}[dd]
			\xys{}@/_/[dlll]_-{\pi'_2}
			\xydot{|-{h' \times I}}[dl]
			\xys{}[dr]^-{\pi'_2}
&\\
I	\xys{}[dd]_(.45){\gamma_3}
&&	F \times_U I	\xyrepi{|(.45){F \times \gamma_2}}[dd]
			\xys{}[rr]_(.7){\pi_2}
			\xys{}[ll]^-{\pi_2}
&&	I	\xyrepi{}[dd]^(.45){\pbkun{\gamma_1,\gamma_2}}
\\
&&&	Z \times X	\xys{|!{[ul];[dl]}{\hole}}@/_/[dlll]_(.75)g
			\xys{}[dr]^-{h \times X}
			\xydot{|-{h' \times X}}[dl]
&\\
B
&&	F \times X	\xys{}[ll]^(.59)e
			\xys{}[rr]_-{\phi \times X}
&&	U \times X
}
\end{equation}
In order to obtain an arrow
$ h' \colon Z \to F $ such that $ e (h' \times X) = g $,
it is enough to show that the square with side $g$ commutes.
Indeed, in this case,
there is $ h' \colon Z \to F $ such that $ \phi h' = h $
by~\eqref{wcc:wedef}.
It follows that
the upper triangle(s) and the square with dotted sides
in diagram~\eqref{wcc:weuniv} commute.
Since $Z \times \gamma_2$ is (regular) epic,
the lower left-hand triangle commutes as well.

To see that the square with side $g$ in diagram~\eqref{wcc:weuniv}
commutes note that,
precomposing its two sides with $q \colon P \repi Z \times_U I$
yields the right-hand square in diagram~\eqref{wcc:rcuniv}.
The claim follows from
the fact that $q$ is (regular) epic.
\end{proof}

\begin{theor} \label{thm:ccex}
Suppose that \catct has right adjoints to weak pullback functors.
If \catct has full diagrams, then \excomct is cartesian closed.
\end{theor}

\begin{proof}
In the terminology of~\cite{Emmenegger2018},
\cref{lem:deslog,lem:wcc} prove that if a category has weak finite limits,
right adjoints to weak pullback functors and full diagrams,
then it has extensional exponentials too.
The statement follows
from Lemma 2.13 in~\cite{Emmenegger2018}.
\end{proof}

Below we collect together the results in this section.

\begin{corol} \label{cor:ccwrap}
Suppose that \catct has right adjoints to weak pullback functors,
and consider the following.
\begin{enumerate}
\item \label{ccwrap:full}
\catct has full diagrams.
\item \label{ccwrap:cc}
\excomct is cartesian closed.
\item \label{ccwrap:wcc}
\catct has weak exponentials.
\end{enumerate}
We have \ref{ccwrap:full} $ \limpd $ \ref{ccwrap:cc} $ \limpd $ \ref{ccwrap:wcc}.
If \catct has binary products,
then \ref{ccwrap:wcc} $ \limpd $ \ref{ccwrap:full}.
\end{corol}

\section{Dependent full diagrams}
\label{sec:dfulldiag}

Recall that \catct denotes a category with weak finite limits
and \excomct its exact completion.
In this section we define an indexed version of full diagrams,
whose existence will endow the internal logic of \catct (hence of \excomct)
with implication and universal quantification.

\begin{defin} \label{def:fulldiag}
Let $ y \colon Y \to X $ and $ x \colon X \to J $
be two arrows in a category with weak finite limits.
A \emph{dependent full diagram over $ x,y $} is a commutative diagram
\begin{equation}\label{depfulldiag}
\xycenterm[C=3em@R=1.5em]{
Y	\xys{}[dr]_-y
&	V	\xys{}[l]_-f \xys{}[d]^-{p_2} \xys{}[r]^-{p_1}
&	U	\xys{}[d]^-u
\\
&	X	\xys{}[r]_-x	&	J	}
\end{equation}
such that the square is a weak pullback and,
for every such diagram $u',f'$ over $y,x$ as below,
there are an arrow $ h \colon U' \to U $,
a weak pullback $ V \toop P \to U' $
of $ V \overset{p_1}{\longrightarrow} U \overset{h}{\longleftarrow} U' $
and an arrow $ k \colon P \to V' $ making the diagram below commute.

\[
\xycenterm[C=1.5em]{
&	V'	\xys{}[dl]_-{f'}
		\xys{|!{[dl];[d]}\hole}[ddr]
		\xys{}@/^1.5em/[rrrr]
&&	P	\xydot{|-k}[ll] \xydot{}[dl] \xydot{}[rr]
&&	U'	\xydot{|-h}[dl] \xys{}@/^1em/[ddl]^-{u'}
\\
Y	\xys{}[drr]_-y
&&	V	\xys{|(.6)f}[ll] \xys{}[d]^-{p_2} \xys{}[rr]^-{p_1}
&&	U	\xys{}[d]_-u
&\\
&&	X	\xys{}[rr]_-x	&&	J	&}
\]
\end{defin}

The same observations as in \cref{rem:fullobj} apply,
mutatis mutandis, to dependent full diagrams.
Moreover, it is not difficult to see that
dependent full diagrams generalise full families of pseudo-relations
from~\cite{EmmeneggerPalmgren2017},
in the sense that the two notions coincide
in well-pointed categories with finite products and weak equalisers.

\begin{rem}\label{rem:axF}
Another category-theoretic version of Aczel's Fullness Axiom
was introduced in~\cite{vdBergMoerdijk2008}
in the context of Algebraic Set Theory~\cite{JoyalMoerdijk1995}.
There one deals with classes of arrows, called \emph{small maps},
in categories which are at least regular.
According to the properties satisfied by the small maps,
various set theories may be interpreted in this structure.
In particular, in order to interpret the Fullness Axiom
in a suitable category \catet
equipped with a class of small maps $\mathcal{S}$,
van den Berg and Moerdijk introduce a condition
$\mathrm{F}(\mathcal{S})$,
called \textbf{(F)} in~\cite[Section 3.7]{vdBergMoerdijk2008}.
By taking the class $\mathcal{S}$ to consist of all arrows of \catet,
condition $\mathrm{F}=\mathrm{F}(\mathrm{Ar}\,\catem)$
makes sense for any regular category.
If a regular category \catet has a projective cover \catpt,
then a straightforward but lengthy computation shows that
F holds in \catet \iff \catpt has dependent full diagrams.
Moreover, Proposition 6.2.5 in~\cite{vdBergMoerdijk2008} entails that
condition F is stable under ex/reg completion.
As a reg/wlex completion is
in particular a regular category with enough projectives~\cite{CarboniVitale1998},
it follows immediately that
$\excomcm \equiv \indxup{(\indxup{\catcm}{reg/wlex})}{ex/reg} $
satisfies F whenever \catct has dependent full diagrams.
\end{rem}

Recall that a \emph{weak dependent product} of two composable arrows
$ Y \overset{y}{\longrightarrow} X \overset{x}{\longrightarrow} J $
in a category with weak finite limits is a commutative diagram as~\eqref{depfulldiag}
such that the square is a weak pullback, $f \colon V \to Y$ is determined by projections
and, for every such diagram $u' \colon U' \to J, f' \colon V' \to Y$ over $y$ and $x$,
there are arrow $U' \to U$ and $V' \to V$ making the obvious diagram commute.

As for the non-indexed case,
as soon as  \catct has pullbacks,
we may regard weak dependent products in \catct
as those dependent full diagrams whose weak pullback is a pullback.

\begin{lem}
If \catct has pullbacks,
weak dependent products are dependent full diagrams
and any dependent full diagram over $y,x$
gives rise to a weak dependent product of $y,x$.
\end{lem}

\begin{proof}
The lemma is proven similarly to \cref{lem:wexpfull}.
\end{proof}

\begin{lem} \label{lem:totfull}
If \catct has dependent full diagrams,
then it has full diagrams.
\end{lem}

\begin{proof}
A full diagram for two objects $ X $ and $ Y $ can be obtained
as a dependent full diagram over $ V \to U \to T $,
where $ T $ is weakly terminal, $ U $ is a weak product of $ X $ and $ T $,
$ V $ is a weak product of $ U $ and $ Y $
and the arrows are the obvious projections.
\end{proof}

\begin{lem} \label{lem:fullsl}
Let \catct be a category with weak finite limits.
\catct has dependent full diagrams \iff every slice of \catct has them.
\end{lem}

\begin{proof}
It follows from the fact that
the forgetful functor $\catcm/J \to \catcm$ preserves and reflects
weak pullbacks and dependent full diagrams.
\end{proof}

\begin{lem} \label{lem:fullrxadj}
If \catct has dependent full diagrams,
then \excomct has right adjoints to inverse images.
\end{lem}

\begin{proof}
Using a choice of dependent full diagrams in \catct
it is possible to define, for every $ f \colon Y \to X $ in \catct,
functors $ \allw_f \colon \posrefm{(\catcm/Y)} \to \posrefm{(\catcm/X)} $.
These functors are right adjoint to weak pullback functors by
the universal property of dependent full diagrams.
The statement follows from \cref{lem:deslog}.
\end{proof}

\begin{theor} \label{thm:lccex}
If \catct has dependent full diagrams,
then \excomct is locally cartesian closed.
\end{theor}

\begin{proof}
It only remains to put together the previous results.
\Cref{lem:fullrxadj} ensures that \excomct has right adjoints to inverse images,
whereas \cref{lem:fullsl,lem:totfull} imply that
$ \catcm/X $ has full diagrams for every $ X $ in \catct.
Hence \cref{thm:ccex} yields the cartesian closure of $ \excomcm/(\exfnctm X) $.
The general statement follows now ``descending'' along a cover
$\exfnctm X \repi A$ as in the proof of
Theorem 3.6 in~\cite{Emmenegger2018}.
\end{proof}

We again collect together the results of this section.

\begin{corol} \label{cor:lccwrap}
Consider the following.
\begin{enumerate}
\item \label{lccwrap:full}
\catct has dependent full diagrams.
\item \label{lccwrap:lcc}
\excomct is locally cartesian closed.
\item \label{lccwrap:wlcc}
\catct has weak dependent products.
\end{enumerate}
We have \ref{lccwrap:full} $ \limpd $ \ref{lccwrap:lcc} $ \limpd $ \ref{lccwrap:wlcc}.
If \catct has pullbacks,
then \ref{lccwrap:wlcc} $ \limpd $ \ref{lccwrap:full}.
\end{corol}

\section{Full diagrams in homotopy categories}
\label{sec:hocat}

In this section we show that, under mild assumptions on a model category \catmt,
the homotopy category \hocatmt has dependent full diagrams
if \catmt has weak dependent products.
Using well-known results,
this implies that the exact completion of the homotopy categories on spaces and CW-complexes
yields locally cartesian closed pretoposes.

$ \overset{\text{\tiny $\sim$}}{\longrightarrow} $
For basic notions on model categories and homotopical algebra we refer to~\cite{Hovey1999}.
Fibrations, cofibrations and weak equivalences are denoted as $ \fibr $,
$ \cofibr $
and $ \wequiv $, respectively.
A path object factorisation for an object $ A $ is denoted as
$ A \wequivand{\cofibr} P A \fibr A \times A $,
and a fibrewise path object factorisation for a fibration $ p \colon A \fibr B $
as $ A \wequivand{\cofibr} P_p A \fibr A \times_B A $.
Since we shall not be concerned here with cylinder objects,
we say that two arrows $ f,g \colon C \to A $ in \catmt are homotopic,
written $ f \htpc g  $, if they are right homotopic
(\ie homotopic \wrt the path object $ P A $).
Similarly, we shall write $ f \htpc_p g $
to mean that they are fibrewise right homotopic over the fibration $ p $.
Note that, since every fibration $p \colon A \fibr B$
is fibrant in the model category structure on $\catmm/B$
induced by that one on \catmt,
fibrewise (right) homotopy is an equivalence relation
on arrows $f,f' \colon X \to A$ such that $p f = p f'$.
If every object in \catmt is cofibrant,
then it is also a congruence.

When every object in a model category \catmt is cofibrant,
the homotopy category \hocatmt is equivalent to the category
obtained quotienting the full subcategory $ \catmm_f $ of \catmt on fibrant objects
by the homotopy relation~\cite{Hovey1999}.
Moreover, \fibcatmt is a category of fibrant objects
in the sense of Brown~\cite{Brown1973}
where, in addition,
every acyclic fibration has a section and where
weak equivalences and homotopy equivalences coincide.
A category of fibrant objects satisfying these additional properties
is called a \emph{path category} by van den Berg and Moerdijk~\cite{vdBergMoerdijk2018}.
More explicitly, a path category may be axiomatised as follows (cf.\ \cite{vdBergMoerdijk2018}).
It has a terminal object and
two classes of distinguished arrows closed under isomorphism and composition,
called weak equivalences and fibrations, such that:
\textit{(i)} weak equivalences are closed under 2-out-of-6,
\textit{(ii)} terminal arrows are fibrations,
\textit{(iii)} pullbacks along fibrations exist and
(acyclic) fibrations are stable under pullback, and
\textit{(iv)} every acyclic fibration has a section.

\begin{defin}[\cite{vdBergMoerdijk2018}, Definition 5.2] \label{def:hwdepprod}
Let $ g \colon B \fibr A $ and $ f \colon A \fibr I $ be two fibrations in a path category \catct. 
A commuting diagram
\[
\xycenterm[C=3em]{
B	\xyfibr{}[dr]_g
&	U \times_I A	\xys{}[l]_-e \xyfibr{}[d] \xyfibr{}[r]
&	U	\xyfibr{}[d]^u
\\
&	A	\xyfibr{}[r]_f	&	I	}
\]
is a \emph{homotopy weak dependent product of $ f $ and $ g $}
if for every such commutative diagram $ u' \colon U' \fibr I $, $ e' \colon U' \times_I A \to B $,
there is $ k \colon u' \to u $ over $ I $ such that $ e (k \times A) \htpc_g e' $.
\end{defin}

Homotopy weak dependent products are called
weak homotopy $\Pi$-types in~\cite{vdBergMoerdijk2018}.
As observed in~\cite{vdBergMoerdijk2018}, 
the (weak) universal property also holds when the arrow $ u' $ is not a fibration.

When the path category is \fibcatmt,
a weak homotopy dependent product arises as fibrant replacement
of a weak dependent product in \catmt.
To prove this fact we need the following result,
which is a reformulation for a model category
of Theorem 2.38 from~\cite{vdBergMoerdijk2018}.

\begin{theor} \label{thm:hdf}
Let \catmt be a model category 
and let $A$ and $B$ be cofibrant objects.
Then every commutative square
\begin{equation} \label{eq:hdf}
\xycenterm{
A	\xys{}[d]^\xywequiv_f \xys{}[r]^k
&	C	\xyfibr{}[d]^g
\\
B	\xys{}[r]_l
& 	D	}
\end{equation}
has a homotopy diagonal filler,
\ie an arrow $ d \colon B \to C $ such that $ g d = l $ and $ d f \htpc_g k $.
Moreover, such a filler is unique up to fibrewise homotopy over $ g $.
\end{theor}

\begin{proof}
We shall first show that every commuting square~\eqref{eq:hdf} has a lower filler,
\ie an arrow $ d $ such that $ g d = l $.
This fact, in turn, allows us to obtain a homotopy witnessing the fact that
the previously constructed lower filler is in fact a homotopy diagonal filler,
and another homotopy witnessing its uniqueness.

Consider a factorisation of $ \pbkun{f,k} \colon A \to B \times_D C $
into an acyclic cofibration $ c \colon A \wequivand{\cofibr} E $
followed by a fibration $ p \colon E \fibr B \times_D C $.
In particular, $E$ is cofibrant.
From \twothree we obtain that $ \pi_1 p \colon E \fibr B $ is an acyclic fibration and,
since $ B $ is cofibrant, it has a section $ s \colon B \wequiv E $.
But then $ d \coloneqq \pi_2 p s \colon B \to C $ is the required lower filler,
as $ g d = l \pi_1 p s = l $.

Therefore every commuting square~\eqref{eq:hdf} has a lower filler.
In particular, we obtain a homotopy $ s \pi_1 p \htpc_{(\pi_1 p)} \id_X $
as a lower filler in
\[
\xycenterm[C=3em]{
B	\xys{}[d]^-\xywequiv_-s \xys{}[r]
&	P_{\pi_1 p} E	\xyfibr{}[d]
\\
E	\xys{}[r]_-{\pbkun{s \pi_1 p, \id_X}}
&	E \times_B E	}
\]
where the top horizontal arrow is $ s $ followed by reflexivity of $ P_{\pi_1 p} E $.
Since $E$ is cofibrant,
it is $ d f = \pi_2 p (s \pi_1 p) c \htpc_g \pi_2 p c = k $.

Finally, given another homotopy diagonal filler $ d' $,
the homotopy witnessing $ d \htpc_g d' $ is obtained as a lower filler in
\[
\xycenterm[C=3em]{
A	\xys{}[d]^\xywequiv_f \xys{}[r]
&	P_g C	\xyfibr{}[d]
\\
B	\xys{}[r]_-{\pbkun{d,d'}}
&	C \times_D C	}
\]
where the top horizontal arrow is the concatenation $ d f \htpc_g k \htpc_g d' f $.
\end{proof}

\begin{rem}
The argument used in the previous proof can be adapted to work in a path category,
so to provide an alternative proof of Theorem 2.38 in~\cite{vdBergMoerdijk2018}.
To this aim it is enough to observe that, in a path category,
the existence of lower fillers for commutative squares as in~\eqref{eq:hdf}
is enough to derive that the homotopy relation is a congruence,
as in the proof of Theorem 2.14 in~\cite{vdBergMoerdijk2018}.
\end{rem}

\begin{corol}
Let \catmt be a right proper model category where every object is cofibrant.
If \catmt has weak dependent products,
then \fibcatmt has homotopy weak dependent products for every pair of composable fibrations.
\end{corol}

\begin{proof}
Let $ f \colon A \fibr I $ and $ g \colon B \fibr A $ be two fibrations in \fibcatmt and let
$ w \colon W \to I $, $ d \colon W \times_I A \to B $ be a weak dependent product of them.
Factor $ w $ as an acyclic cofibration $ c \colon W \wequivand{\cofibr} U $ followed by a fibration $ u \colon U \fibr I $.
Since \catmt is right proper, $ W \times_I A \to U \times_I A $ is also a weak equivalence,
hence we obtain $ e \colon U \times_I A \to B $ as homotopy diagonal filler.

The required universal property is depicted in the diagram below
\[
\xycenterm[C=2.5em@R=1.5em]{
&&	U' \times_I A	\xydot{|-{k \times A}}[dl]
			\xys{}@/_3em/[ddll]_(.7){e'}
			\xyfibr{|!{[dl];[dr]}\hole|!{[ddl];[dd]}\hole}@/^1em/[dddl]
			\xyfibr{}[rr]
&&	U'	\xydot{|-k}[dl] \xyfibr{}@/^1em/[dddl]^-{u'}
\\
&	W \times_I A	\xys{}[dl] \xys{}[d]^\xywequiv \xyfibr{}[rr]
&&	W	\xycofibr{}[d]^\xywequiv_-c
&\\
B	\xyfibr{}[dr]_-g
&	U \times_I A	\xys{|-{\,e\,}}[l] \xyfibr{}[d] \xyfibr{}[rr]
&&	U	\xyfibr{}[d]_-u
&\\
&	A	\xyfibr{}[rr]_-f	&&	I	&}
\]
where $ e (ck \times A) \htpc_g e' $ since $ e $ is just a homotopy diagonal filler.
\end{proof}

Homotopy weak dependent products also enjoy another universal property
with respect to certain homotopy diagrams.
This is proved below in \cref{lem:wdepprodhfull}
and it is a consequence of the following result.

\begin{prop}[\cite{vdBergMoerdijk2018}, Proposition 2.31] \label{prop:strict}
Let \catct be a path category and let
\[
\xycenterm{
						&	A	\xyfibr{}[d]^f	\\
	C	\xys{}[r]_g \xys{}[ur]^k	&	B			}
\]
be a diagram that commutes up to homotopy.
Then there is $ k' \colon C \to A $ such that $ k' \htpc k $ and $ f k' = g $.
\end{prop}

\begin{rem} \label{rem:strict}
\Cref{prop:strict} has the important consequence
that pullbacks in \fibcatmt along fibrations are homotopy pullbacks
and so are mapped to weak pullbacks in \hocatmt.
\end{rem}

\begin{defin} \label{def:hfulldiag}
Let $ f \colon A \to I $ and $ g \colon B \to A $ be two arrows in a path category \catct.
A diagram
\[
\xycenterm[C=3em@R=1.5em]{
B	\xys{}[dr]	&	V	\xys{}[l] \xys{}[d] \xys{}[r]
&	U	\xys{}[d]
\\
&	A	\xys{}[r]	&	I,	}
\]
that commutes up to homotopy
and where the square is a homotopy pullback, is \emph{homotopy full over $ f,g $}
if, for every such diagram over $ B \to A \to I $ commuting up to homotopy,
there are an arrow $ U' \to U $, a homotopy pullback $ V \toop P \to U' $ of $ V \to U \toop U' $
and an arrow $ P \to V' $ making the diagram below commute up to homotopy.
\[
\xycenterm[C=3em@R=1.5em]{
&	V'	\xys{}[dl] \xys{|!{[dl];[d]}\hole}@/_1.5em/[dd] \xys{}@/^1.5em/[rr]
&	P	\xydot{}[l] \xydot{}[dl] \xydot{}[r]
&	U'	\xydot{}[dl] \xys{}@/^1em/[ddl]
\\
B	\xys{}[dr]	&	V	\xys{}[l] \xys{}[d] \xys{}[r]
&	U	\xys{}[d]
&\\
&	A	\xys{}[r]	&	J	&}
\]
\end{defin}

\begin{rem} \label{rem:hfulldiag}
Let \catct be a path category.
Since \hocatct is the quotient of \catct by the homotopy relation
(cf.\ Theorem 2.16 in~\cite{vdBergMoerdijk2018}),
the image in \hocatct of a homotopy full diagram
over $ f,g $ is a full diagram over $ [f],[g] $.
\end{rem}

\begin{lem} \label{lem:wdepprodhfull}
Let \catct be a path category
and let $ f \colon A \fibr I $ and $ g \colon B \fibr A $ be two fibrations.
A homotopy weak dependent product of $ f $ and $ g $ is a homotopy full diagram over $ f,g $.
\end{lem}

\begin{proof}
Let $ u \colon U \fibr I $, $ e \colon U \times_I A \to B $
be a homotopy weak dependent product of $ f $ and $ g $.
\Cref{rem:strict} implies that $ U \times_I A $ is a homotopy pullback.

Let now 
\[
\xycenterm[C=3em@R=1.5em]{
B	\xyfibr{}[dr]_g
&	V'	\xys{}[l]_{e'} \xys{}[d]^{v_2} \xys{}[r]^{v_1}
&	U'	\xys{}[d]^{u'}
\\
&	A	\xyfibr{}[r]_f	&	I,	}
\]
be commutative up to homotopy and such that the square is a homotopy pullback.
Hence there is an arrow $ \psi \colon U' \times_I A \to V' $
such that $ v_1 \psi \htpc \pi'_1 $ and $ v_2 \psi \htpc \pi'_2 $.
In particular, the diagram below commutes up to homotopy
\begin{equation}
\xycenterm{
U' \times_I A	\xys{}[dr]_{\pi'_2} \xys{}[r]^-{\psi}
&	V'	\xys{}[r]^{e'}	&	B	\xyfibr{}[dl]^g
\\
&	A	&}
\end{equation}
and \cref{prop:strict} implies that there is $ h \colon U' \times_I A \to B $ that makes the above triangle commute
and which is homotopic to $ e' \psi $.

The universal property of the homotopy weak dependent product then yields an arrow
$ k \colon U' \to U $
such that everything in the diagram below commutes strictly except for the two top-left triangles with common edge $ h $,
which only commute up to homotopy.

\[
\xycenterm[C=3em]{
&	V'	\xys{}[dl]_{e'}
&	U' \times_I A	\xys{}[l]_-\psi \xys{|-{\,h\,}}[dll] \xydot{}[dl] \xys{|!{[dl];[d]}\hole}@/^1em/[ddl] \xyfibr{}[r]
&	U'	\xydot{|-k}[dl] \xys{}@/^1em/[ddl]^-{u'}
\\
B	\xyfibr{}[dr]_-g
&	U \times_I A	\xys{}[l]^-e \xyfibr{}[d] \xyfibr{}[r]
&	U	\xyfibr{}[d]_-u
&\\
&	A	\xyfibr{}[r]_-f	&	I	&}
\]

Hence, as required, the square with two dotted sides above is a homotopy pullback and the diagram below commutes up to homotopy.

\[
\xycentermqed[C=4em]{
&	V'	\xys{}[dl]_{e'}
		\xys{|!{[dl];[d]}{\hole}}@/_2.5em/[dd]^(.2){v_2}
		\xys{}@/^1.5em/[rr]^-{v_1}
&	U' \times_I A	\xys{}[l]^(.6){\psi} \xys{}[dl] \xyfibr{}[r]
&	U'	\xys{|-{\,k\,}}[dl] \xys{}@/^1em/[ddl]^-{u'}
\\
B	\xyfibr{}[dr]_-g
&	U \times_I A	\xys{|-{\,e\,}}[l] \xyfibr{}[d] \xyfibr{}[r]
&	U	\xyfibr{}[d]_-u
&\\
&	A	\xyfibr{}[r]_-f	&	I	&}
\]
\end{proof}

\begin{theor} \label{thm:lccexh}
Let \catmt be a right proper model category where every object is cofibrant.
If \catmt has weak dependent products,
then \hocatmt has dependent full diagrams and,
in turn, $ \excomm{(\hocatm \catmm)} $ is locally cartesian closed.
\end{theor}

\begin{proof}
\Cref{lem:wdepprodhfull} and \cref{rem:hfulldiag} yield a full diagram over $ [f],[g] $
whenever $ f $ and $ g $ are both fibrations.
Since arrows in \catmt factor as weak equivalences and fibrations
and the former are isomorphisms in \hocatmt,
this is enough to conclude that \hocatmt has dependent full diagrams.
The last statement is an application of \cref{thm:lccex}.
\end{proof}

As an application of \cref{thm:lccexh},
consider the two standard model structures on the category of topological spaces
by Quillen~\cite{Quillen1967} and by Str\o{}m~\cite{Strom72},
which we denote by $ \cattopm_Q $ and $ \cattopm_S $, respectively.
The latter is right proper and every space is cofibrant.
Furthermore, Carboni and Rosolini showed
that \cattopt has weak dependent products~\cite{CarboniRosolini2000}.
Therefore $ \excomm{(\hocatm \cattopm_S)} $ not only is a pretopos,
as proved in~\cite{GranVitale1998},
but it is also locally cartesian closed.
This answers a question left open by Gran and Vitale in~\cite{GranVitale1998}.
In addition, although in $ \cattopm_Q $ the cofibrant objects are just the CW-complexes,
$ \cattopm_Q $ is Quillen equivalent to simplicial sets with the Quillen model structure.
This latter category does satisfy the hypothesis of our theorem,
therefore $ \excomm{(\hocatm \cattopm_Q)} $ is locally cartesian closed too.

\section*{Acknowledgements}
I am in debt with the late Erik Palmgren
for sharing with me his idea of using
the Fullness Axiom to obtain local cartesian closure.
The results in this paper were presented
at the 5th Workshop on Categorical Methods in Non-Abelian Algebra,
Louvain-la-Neuve, June 1-3 2017,
and at the International Category Theory Conference in Vancouver, July 16-22 2017.
I gratefully thank the organisers of both events for giving me the opportunity to speak.
This paper was partly written while I was visiting
the Hausdorff Research Institute for Mathematics in Bonn
during May 2018 in occasion of the trimester program Types, Sets and Constructions.
I thank the Institute for providing an excellent working environment.
Support from the Royal Swedish Academy of Sciences is also acknowledged.
Finally, I would like to express my gratitude to the anonymous referee for
useful comments on a previous draft of the paper.


\end{document}